\documentclass[reqno,a4paper,12pt]{amsart}  
 \usepackage{pstricks,pst-node,pst-coil,pst-plot}  
\usepackage{amsfonts}
\usepackage{stmaryrd}
 \usepackage{hyperref}                           
 \usepackage{indentfirst}
\usepackage{eufrak}
\usepackage{mathrsfs}
\usepackage[pdftex]{graphicx}

\theoremstyle{plain}      
 
\newtheorem{thm}{Theorem}[section]     
\newtheorem{theorem}[thm]{Theorem}

\newtheorem{lemma}[thm]{Lemma}     
     
\newtheorem{proposition}[thm]{Proposition}     
     
\theoremstyle{remark}

\theoremstyle{definition}      
     
\newtheorem{definition}[thm]{Definition}     

\DeclareMathAlphabet{\doba}{U}{msb}{m}{n}

\def\Arg{{\mathop{\rm Arg}}}
\def\Log{{\mathop{\rm Log}}}

\newcommand{\Ree}{\mathop{\rm Re}\nolimits}

\let\<\langle 
\let\>\rangle

\newcommand{\definedas}{\mathrel{\raise.095ex\hbox{\rm :}\mkern-5.2mu=}}

\begin{document}     



\title[Generalized  logarithmic Gauss map]{Generalized  logarithmic Gauss map and its relation to (co)amoebas}
 
\author{Farid Madani} 
\address{Fakult\"at f\"ur  Mathematik \\ 
Universit\"at Regensburg \\
93040 Regensburg \\  
Germany}
\email{farid.madani@mathematik.uni-regensburg.de}

\author{Mounir Nisse} 
\address{Department of Mathematics, Texas A\&M University, College Station, TX 77843-3368, USA.}
\email{nisse@math.tamu.edu}
\thanks{Research of  the second author is partially supported by NSF MCS grant DMS-0915245.}
\begin{abstract}
We define the generalized logarithmic Gauss map for algebraic varieties of the complex algebraic torus of any codimension. Moreover,   we describe the set of  critical points of the logarithmic mapping restricted to our variety,  and we show  an analogous of Mikhalkin's result on the critical points of the logarithmic map  restricted to a  hypersurfaces.

\end{abstract}

\subjclass[2000]{}

\date{\today}


\maketitle


\section{Introduction} 

The Gauss map for hypersurfaces in the complex algebraic torus $(\mathbb{C}^*)^n$ is called  {\em the  logarithmic Gauss map}. The set of its  critical points  of the logarithmic map restricted to an algebraic variety  in the complex torus plays  an important role in real algebraic geometry. Indeed,  in the case of hypersurfaces, it contains the real part of the hypersurface itself. Moreover, if the hypersurface is real then they coincide (see \cite{M1-00} or \cite{M2-02} for more details).   We define the generalized logarithmic Gauss map for varieties in the complex algebraic torus of any codimension, and we describe  the set of  critical points of the logarithmic map restricted to our variety. Mikhalkin shows that the set of critical points of the logarithmic   map  in the case of hypersurface, coincide with the inverse image  the logarithmic Gauss map of the real projective space $\mathbb{P}\mathbb{R}^{n-1}$ . We show the analogous of Mikhalkin's result for algebraic  varieties of higher codimension.

\section{Preliminaries}

Let us  define  the logarithmic Gauss map of a hypersurface in the complex
algebraic torus (see Kapranov  \cite{K1-91} for more details).
\begin{definition}
Let $(\mathbb{C}^*)^n$ be the algebraic torus and $V\subset
(\mathbb{C}^*)^n$ be an algebraic hypersurface. For each
$z\in(\mathbb{C}^*)^n$ let $l_z :  h\longmapsto z.h$ be the translation
by $z$. The {\em Gauss map} of the hypersurface $V$ is the rational map
$\gamma : V \longrightarrow \mathbb{CP}^{n-1}$ taking a smooth point
$z\in V$ to the hyperplane $d(l_{z^{-1}})(T_zV)\subset
T_1((\mathbb{C}^*)^n)$, where $ \mathbb{CP}^{n-1} =
\mathbb{P}(T_1((\mathbb{C}^*)^n))$.
\end{definition}

\vspace{0.2cm}

 Let $\gamma$ be the map defined as follows
$$
\gamma  : V \longrightarrow  \mathbb{CP}^{n-1},
$$

$$
z \longmapsto T_zV\subset
T_z((\mathbb{C}^*)^n)\stackrel{t^*_{z^{-1}}}{\longrightarrow}
t^*_{z^{-1}}(T_zV)\subset T_1((\mathbb{C}^*)^n \longrightarrow   \mathbb{CP}^{n-1}.
$$
Hence, $\gamma$ is defined as  follows
$$
 V_{reg}  \ni z \longmapsto   \gamma (z)  =  [t^*_{z^{-1}}(T_zV)] \in  \mathbb{CP}^{n-1},
$$

\noindent where $[t^*_{z^{-1}}(T_zV)]$ is the class of $t^*_{z^{-1}}(T_zV)$ in $
\mathbb{CP}^{n-1}$. The map $\gamma$ is called {\em the logarithmic Gauss map}.

We suppose that $V\subset (\mathbb{C}^*)^n$ is  a complex algebraic hypersurface
defined by a polynomial $f$ and nowhere singular. Let $z\in V$,
$U\subset (\mathbb{C}^*)^n$ be a small open neighborhood of $z$ and $\mathscr{L}og_U :
U\longrightarrow \mathbb{C}^n$ be a choice of a  branch of   the holomorphic logarithm function
 and then we take the image of $U\cap V$ by $\mathscr{L}og_U$. Hence, the
vector $v (z)\in \mathbb{C}^n$ (viewed as
$T_{\mathscr{L}og_U(z)}(\mathbb{C}^n$) tangent to $\mathscr{L}og_U(U\cap V)$ at the point
  $\mathscr{L}og_U(z)$ is a complex vector, which is independent of the choice
  of the holomorphic branch. So, it defines a map from $U\cap V$ (a
  rational map from  all $V$) to $\mathbb{CP}^{n-1}$, called the
  {\em logarithmic Gauss map} of the hypersurface $V$. Explicitly one has:

\begin{eqnarray*}
\gamma :&V&\longrightarrow\,\,\,\mathbb{CP}^{n-1}\\
&z&\longmapsto\,\,\, \gamma (z) =  [v (z)],
\end{eqnarray*}

\noindent where $[v (z)]$ denotes the class of the vector $v
(z)$ in $\mathbb{CP}^{n-1}$. We decompose the identity as follow:
$$
U\cap
V\stackrel{\mathscr{L}og_U}{\longrightarrow}\mathbb{C}^n\stackrel{\exp}{\longrightarrow}(\mathbb{C}^*)^n\stackrel{f}{\longrightarrow}\mathbb{C}.
$$
Let $\Log$ be the logarithmic mapping defined as follows
$$\
Log (z_1,\ldots ,z_n) =
(\log |z_1|, \ldots ,\log |z_n|).
$$
Now, let us consider the amoeba $\mathscr{A}$ of the curve $V$ which the image of $V$ under the logarithmic mapping.
More precisely, we  consider  $\Log (U\cap V)$.
It is  clear that $\Log
(U\cap V) = \Ree \circ \mathscr{L}og_U (U\cap V)$ where $\Ree (z_1,\ldots ,z_n) =
(\Ree (z_1),\ldots , \Ree (z_n))$, and $\Ree$ denotes the real part. So, if
$t=\mathscr{L}og_U(u)$ where $u = (u_1,\ldots , u_n)\in U$, then $dt = \frac{du}{u}$ which is  $(dt_1, \ldots ,dt_n)$, and we get:
\begin{eqnarray*}
v (z)& =& \frac{df(\exp (t))}{dt}_{\mid t=\mathscr{L}og_U(z)}\\
&=&(u_1\frac{\partial f}{\partial u_1}(u), \ldots , u_n\frac{\partial f}{\partial u_n}(u))_{\mid u=z}\\
&=&(z_1\frac{\partial f}{\partial z_1}(z),\ldots ,  z_n\frac{\partial f}{\partial z_n}(z))
\end{eqnarray*}
\noindent Hence, a point $z\in V$ is a critical point of the function
$\Log = \Ree\circ\mathscr{L}og_U$ is equivalent to say that the tangent space
of $\mathscr{L}og_U (U\cap V)$ at the point $\mathscr{L}og_U (z)$ contains at least $n-1$ non zero purely 
imaginary  linearly independent  vectors $v_i$, which means that $T_{\mathscr{L}og_U (z)}(\mathscr{L}og_U (U\cap
V))$ is the complexification of a real hyperplane, and then it is invariant under complex conjugation so
 $[v (z)] = \gamma (z) \in
\mathbb{RP}^{n-1}\subset \mathbb{CP}^{n-1}$.

\vspace{0.3cm}

\section{Generalized  logarithmic Gauss map}

We generalize the logarithmic Gauss map for a complex algebraic varieties of any codimension.
We denote by $\mathbb{G}_{r,\, n}$ the  Grassmannian of all complex $r$-dimentional planes through the origin in $\mathbb{C}^n$ 
(i.e., the complex Grassmann manifold).
Let $V\subset (\mathbb{C}^*)^n$ be an algebraic variety of dimension $k$ with defining ideal $\mathcal{I}(V)$ 
generated by $\{ f_1,\ldots ,f_l\}$. We define a holomorphic map $\gamma_G$ from the set of smooth points of $V$ to
 the complex Grassmannian $\mathbb{G}_{n-k,\, n}$ which is a generalization of the logarithmic Gauss map in case of hypersurfaces.
 We denote by $V_{reg}$ the subset of smooth points of $V$, and $M(l\times n)$ be the set of $l\times n$ matrices. Let 
$g_G$ be the following map:
\[
\begin{array}{ccccl}
g_G&:&V_{reg}&\longrightarrow&M(l\times n)\\
&&z=(z_1,\ldots ,z_n)&\longmapsto&
\left( \begin{array}{ccc} z_1\frac{\partial f_1}{\partial z_1}(z) & \ldots & z_n\frac{\partial f_1}{\partial z_n}(z) \\ 
\vdots & \vdots& \vdots \\ 
 z_1\frac{\partial f_l}{\partial z_1}(z)& \ldots &  z_n\frac{\partial f_l}{\partial z_n}(z)
 \end{array} \right) .
\end{array}
\]
 Since $z$ is a smooth point of $V$, it is clear that the complex vector space $L_z$ generated by the rows of  the matrix $g_G(z)$ is of dimension $n-k$,
 and orthogonal to the tangent space to $V$ at $z$. Indeed, locally  $V$ is a complete intersection in an open neighborhood $U$  containing $z$.
 In other words, the image of $V_{reg}$ by $g_G$ is contained in the subvariety of 
$M(l\times n)$ consisting of $l\times n$ matrices of rank $n-k$, which we identify to the complex Grassmannian $G_{\mathbb{C}}(n-k, n)$. Composing this identification with $g_G$
  we obtain a map :
  $$
  \gamma_G : V_{reg}\rightarrow \mathbb{G}_{n-k,\, n}
  $$ 
  which we call the {\em generalized logarithmic Gauss map}.
  
  \vspace{0.2cm}

If $V\subset (\mathbb{C}^*)^n$ is a hypersurface,  Mikhalkin proves that the set of critical points of $\Log_{|V}$ coincides with 
$\gamma_G (\mathbb{R}P^{n-1})$ (see Lemma 3. \cite{M1-00}). In this section we generalize this fact for varieties of any codimension.
Let $z\in V_{reg}$ and $\mathscr{L}og$ be a branch of a holomorphic logarithm:
$$
(z_1,\ldots ,z_n)\mapsto (\log (z_1),\ldots , \log (z_n))
$$ 
defined in a neighborhood of $z$. 

\vspace{0.2cm}

\begin{lemma}
 With the above  notations, a point $z\in V_{reg}$ is critical for $\Log_{| V}$ if and only if the image under $d\mathscr{L}og$  of the tangent space $T_zV$ to $V$ at $z$   contains at least $s$ linearly independent purely imaginary vectors with $s = \max \{ 1, 2k-n+1\}$.
\end{lemma}

\begin{proof}
 A point $z\in V_{reg}$ is critical for $\Log_{| V}$ is  equivalent to the fact that $V$ and the orbit of the real torus
 $(S^1)^n$ are  not transversal at $z$. Indeed, $\mathscr{L}og$ takes the tangent space to an orbit of $(S^1)^n$ to a translate of 
$\sqrt{-1}\mathbb{R}^n$ in $\mathbb{C}^n$. Therefore, $z$ is critical for $\Log_{| V}$ is equivalent to say that 
$d\mathscr{L}og (T_zV)$ contains at least $s$ linearly independent purely imaginary vectors with $s+ 2(k-s)\leq \min \{ n-1, 2k-1\}$.
 If $Re$ denotes the real part of $\mathbb{C}^n$ coordinatewise, this is equivalent to the fact that
 the map $Re$ from $\mathscr{L}og (V)$ to $\mathbb{R}^n$ is not a submersion. Hence, $s$ should be at least $\max \{ 1, 2k-n+1\}$.
\end{proof}

Let $F = \mathbb{R}^n$ be the real part of  $\mathbb{C}^n$, and  $\sigma_m$  be the Schubert cell:
$$
\sigma_m := \{ E\in \mathbb{G}_{n-k,\, n} \, |\, \dim (E\cap F) = m\} .
$$
The dimension of $\sigma_m$ is equal to the dimension of the real Grassmannian $G(m,n)$ plus the real dimension of the complex 
Grassmann manifold $\mathbb{G}_{k-m,\, n-m}$ which is equal to $m(n-m) + 2(k-m)(n-m)$.

\vspace{0.2cm}
 
\begin{theorem}\label{main theorem}
 Let $V\subset (\mathbb{C}^*)^n$ be an algebraic variety of dimension $k$, and $\textrm{Critp}(\Log_{| V})$ be the set
 of critical points of the restriction of the logarithmic mapping to $V$. Then we have the following:
\begin{itemize}
\item[(i)]\, If $n\geq 2k$, then
$$
\textrm{Critp}(\Log_{| V}) = \bigcup_{j=1}^k\gamma_G^{-1}(\sigma_{n-2k+j});
$$
\item[(ii)]\, If $n < 2k$, then
$$
\textrm{Critp}(\Log_{| V}) = \bigcup_{j=1}^{n-k}\gamma_G^{-1}(\sigma_{j}).
$$
\end{itemize}
\end{theorem}

\begin{proof}
First of all, the statements of Theorem  \ref{main theorem} are  also valid for the set of critical points of the argument map $\textrm{Critp}(\Arg_{| V})$ because the set of critical points of $\textrm{Critp}(\Log_{| V})$ and $\textrm{Critp}(\Arg_{| V})$ are the same (see for example  Corollary 3.1 in \cite{MN-11}).

Let $z$ be a critical point of $\Log_{| V}$, and $n\geq 2k$. Then, the tangent space to $V$ at $z$ contains at least $j$ purely 
imaginary  tangent vector with $1\leq j\leq k$. This means that $T_zV$ contains a real subspace $L_z$ of dimension at least $j$. So,
 the normal space $L_z^{\perp}$ to $L_z$ which is  $\mathbb{C}^{n-j}$ contains the normal space $\gamma_G(z) = T_z^{\perp}V$ to $T_zV$.
  Hence the expected dimension of $T_z^{\perp}V\cap \overline{T_z^{\perp}V}$ is $2(n-k)-(n-j) = n-2k+j$, and the first statement of 
 the theorem is done. In the case where $n<2k$, then $T_zV$ contains at least $2k-n+j$ purely imaginary linearly independent  tangent vectors with
 $1\leq j\leq n-k$. This means that the tangent space to $V$ at $z$ contains a real subspace $L_z$ of dimension at least $2k-n+j$,
 and its normal is $(\mathbb{C}^*)^{n-(2k-n+j)}$. Hence, the expected dimension of $T_z^{\perp}V\cap \overline{T_z^{\perp}V}$
 is $2(n-k) - (n-(2k-n+j)) = j$, and the second part of the theorem is done.
\end{proof} 

\begin{proposition}
Let $\mathscr{P}(k)\subset (\mathbb{C}^*)^n$ be a generic $k$-dimensional affine linear space with $n\geq 2k$. Suppose that the complex dimension of $\mathscr{P}(k)\cap \overline{\mathscr{P}(k)}$ is equal to $l$, with $0\leq l\leq k$. Then, for any regular value $x$ in the amoeba $\mathscr{A}(\mathscr{P}(k))$,  the cardinality of $\Log^{-1}(x)$ is at least $2^l$.
\end{proposition}

\begin{proof}
 If $\dim_{\mathbb{C}}(\mathscr{P}(k)\cap \overline{\mathscr{P}(k)}) = l$, then the intersection $\mathscr{P}(k)\cap \mathbb{R}^{2k}$ where  $\mathbb{R}^{2k}$  is the real part of  $\mathbb{C}^{2k}$ is an affine linear space  of real dimension $l$. We use now, induction on $k$.
 Suppose $k=1$, it is known that any regular value of $\mathscr{A}(\mathscr{P}(1))$ is covered twice by the logarithmic mapping (see Remark 7.1 in  \cite{MN-11}). If $k >1$, and $l < k$,  then $\mathscr{P}(k)$ contains a real  $l$-dimensional affine linear space  $\mathscr{P}(l)$, and  $\mathscr{P}(k)$ can be seen as the cartesian product of $\mathscr{P}(l)$ and some complex $(k-l)$-dimensional affine linear space $\mathscr{P}'$. Using induction, any regular value of $\mathscr{A}(\mathscr{P}(l))$ is covered at least $2^l$ by the logarithmic mapping. So, the  regular values of  $\mathscr{A}(\mathscr{P}(k))$ are  also covered at least $2^l$. If $l=k$, then $\mathscr{P}(k)$ contains a real  $(k-1)$-dimensional affine linear space  $\mathscr{P}(k-1)$, and a real line $L$ such that $\mathscr{P}(k)$ can be viewed  as the cartesian product of $\mathscr{P}(k-1)$  and $L$. Using induction, we obtain the result.
 \end{proof}
 
 \vspace{0.3cm}
 
 {\bf Question} Kapranov  classified algebraic hypersurfaces $V\subset (\mathbb{C}^*)^n$ such that $\gamma  : V \longrightarrow  \mathbb{CP}^{n-1}$ is a birational isomorphism \cite{K1-91}. He prove that such varieties are the reduced $A$-discriminant for some finite subset $A$ of $\mathbb{Z}^n$ such that $A$ affinely generates  over $\mathbb{Z}$ the lattice 
$\mathbb{Z}^n$. It was shown by M. A. Cueto and A. Dickenstein  \cite{CD-06} that this last assumption can be removed.
Let $V\subset (\mathbb{C}^*)^n$ be an irreducible algebraic variety of dimension $k$, and $\gamma_G  : V \longrightarrow  \gamma_G(V)\subset \mathbb{G}_{n-k,\, n}$.

\vspace{0.3cm}

{\it Classify  the $k$-dimensional    algebraic varieties of the complex algebraic torus  such that $\gamma_G$ is a birational isomorphism. Is there a notion of a  multi-discriminant  i.e.,  an  $(A_1,\ldots , A_m)$-discriminant?} 

See \cite{GKZ-94} for more details about the $A$-discriminant.

\end{document}